\documentclass[12pt]{amsart}
\usepackage{amsmath}
\usepackage[psamsfonts]{amssymb}
\usepackage{amsfonts}
\usepackage{algorithm2e}
\usepackage{graphicx}
\usepackage{verbatim}
\usepackage{dsfont}
\usepackage{tikz-cd}
\usepackage{mathrsfs}
\usepackage{color}
\usepackage[english]{babel}
\usepackage{fontenc}
\usepackage{indentfirst}
\usepackage{tikz}
\usetikzlibrary{matrix,arrows,decorations.pathmorphing}
\usepackage{algorithm2e}
\usepackage[all]{xy}
\usepackage[T1]{fontenc}
\usepackage{hyperref}
\usepackage{dsfont}

\theoremstyle{definition}
\newtheorem{theorem}{Theorem}[section]
\newtheorem{prop}[theorem]{Proposition}
\newtheorem{lemma}[theorem]{Lemma}

\newtheorem{cor}[theorem]{Corollary}

\newtheorem{conj}[theorem]{Conjecture}

\newtheorem{ex}[theorem]{Example}

\newtheorem{dfn}[theorem]{Definition}

\newtheorem{remark}[theorem]{Remark}

\def\ep{\epsilon}

\def\R{\mathbb{R}}
\def\Z{\mathbb{Z}}

\def\Aut{{\rm Aut}}

\def\Diff{{\rm Diff}}

\def\Sympeo{{\rm Sympeo}}
\def\Symp{{\rm Symp}}
\def\Aut{{\rm Aut}}
\def\Cont{{\rm Cont}}
\newcommand{\cL}{\mathcal{L}}

\begin{document}

\title[Chekanov's dichotomy in contact topology]{Chekanov's dichotomy in contact topology}

\noindent
\author{Daniel Rosen}
\email{daniel.rosen@rub.de}
\address{Fakult\"{a}t f\"{u}r Mathematik, Ruhr-Universit\"{a}t Bochum\\ Universit\"{a}tstr. 150, 44780 Bochum, Germany}

\author{Jun Zhang}
\email{jun.zhang.3@umontreal.ca}
\address{Centre de recherches mathématiques, University of Montreal, \\ C.P. 6128 Succ. Centre-Ville, Montreal, QC, H3C 3J7, Canada}

\begin{abstract}
In this paper we study submanifolds of contact manifolds. The main submanifolds we are interested in are contact coisotropic submanifolds. They can be viewed as analogues to symplectic contact coisotropic submanifolds, and can be defined by the symplectic complement with respect to the symplectic structure $d\alpha|_{\xi}$, the restriction of $d\alpha$ on the contact hyperplane field $\xi$. Based on a correspondence between symplectic and contact coisotropic submanifolds, we can show contact coisotropic submanifolds admit a $C^0$-rigidity, similar to Humili\`ere-Leclercq-Seyfaddini's coisotropic rigidity on symplectic manifolds in \cite{HLS15}. Moreover, based on Shelukhin's norm in \cite{She17} defined on the contactomorphism group, we define a Chekanov type pseudo-metric on the orbit space of a fixed submanifold of a contact manifold. Moreover, we can show a dichotomy of (non-) degeneracy of this pseudo-metric when the dimension of this fixed submanifold is equal to the one for a Legendrian submanifold. This can be viewed as a contact topology analogue to Chekanov's dichotomy in \cite{Che00} of (non-)degeneracy of Chekanov-Hofer's metric on the orbit space of a Lagrangian submanifold. The proof of our result follows several arguments from \cite{Ush14} and \cite{Ush15}. \end{abstract}

\maketitle 
\tableofcontents

\section{Introduction and statement of results}
Using Hofer's metric to study submanifolds of a symplectic manifold has been carried out in \cite{Ush14}, extending Chekanov's work in \cite{Che00}. Chekanov's original work focused on Lagrangian submanifolds in symplectic topology, whereas Usher considered, more generally, coisotropic submanifolds. It is natural to extend this story to contact topology. By analogy, the main submanifolds of a contact manifold that we are interested in are {\it contact} coisotropic submanifolds. They are defined as follows.

\begin{dfn} (Definition 2.1 in \cite{Hua15}) \label{dfn-Hua15} Let $(M^{2n+1}, \xi = \ker \alpha)$ be a co-oriented contact manifold. A submanifold $Y \subset M$ is called {\it contact} coisotropic if for any point $p \in Y$, $(T_p Y \cap \xi_p)^{\perp_{d\alpha}} \subset T_p Y \cap \xi_p$ where $\perp_{d\alpha}$ denotes the symplectic orthogonal complement with respect to the non-degenerate $2$-form $d \alpha|_{\xi}$ on $\xi$. \end{dfn}

\begin{remark}
Let up points out that this definition is independent of the choice of the contact form $\alpha$, since the conformal class of the symplectic form $d\alpha$ on the contact distribution $\xi$ is well-defined (i.e., independent of $\alpha$).
\end{remark}

A crucial criterion to check whether a submanifold is contact coisotropic or not is the following result that is analogous to the symplectic case. Recall that a submanifold $N$ of a symplectic manifold $(W,\omega)$ is symplectic coisotropic if and only if the ideal $I_N = \{F \in C_c^{\infty}(W) \,| \, F|_N = 0\}$ is a Lie subalgebra of $C^{\infty}_c(M)$ under the symplectic Poisson bracket (see Chapter I Section 2, Lemma 2.1 in \cite{KM93}).  In terms of the {\it contact Poisson bracket} $\{-, -\}_{\alpha}$, whose definition is recalled in Section \ref{sec-prelim} below, we have the following similar result. 

\begin{prop} \label{C-P} Let $(M, \xi = \ker \alpha)$ be a co-oriented contact manifold. A submanifold $Y \subset M$ is contact coisotropic if and only if the ideal $I_Y = \{F \in C_c^{\infty}(M) \,| \, F|_Y = 0\}$ is a Lie subalgebra with respect to the {contact} Poisson bracket $\{-, - \}_{\alpha}$. \end{prop}

Proposition \ref{C-P} has useful corollaries. On the one hand, it establishes various correspondences between symplectic coisotropic submanifolds and contact coisotropic submanifolds, see Section \ref{sec-cor}. On the other hand, based on the main result from \cite{HLS15}, we obtain a $C^0$-rigidity property of contact coisotropic submanifolds as follows.

\begin{theorem} \label{rigidity} Suppose $Y$ is a coisotropic submanifold of a contact manifold $(M, \xi = \ker \alpha)$ and $\phi \in \Aut(M, \xi)$. If $\phi(Y)$ is smooth, then $\phi(Y)$ is also coisotropic. \end{theorem}

Here $\Aut(M, \xi)$ is defined in Definition \ref{aut-group} where roughly speaking each element is a homeomorphism that is a $C^0$-limit of contactomorphisms (in a strong sense). This can be regarded as a contact topology analogue to the group ${\rm Hameo}(M, \omega)$ that is often used in $C^0$-symplectic geometry.

\begin{ex}\label{ex-intro-1} Let $Y$ be a Legendrian knot of a compact contact 3-manifold $(M, \xi)$. For any $\phi \in \Aut(M, \xi)$, if $\phi(Y)$ is smooth, then it is also a Legendrian knot. \end{ex}

\begin{remark} It is a well-known fact (cf. Theorem 2.5 in \cite{Etn05}) that any knot, not necessarily Legendrian, can be $C^0$-approximated by Legendrian knots via adding more and more ``zigzags''. 
	In view of Example \ref{ex-intro-1}, this approximation cannot be of the form $\phi_k(Y)$ for a Legendrian knot $Y$ and a sequence $\{\phi_k\}_k$ of contactomorphisms converging in the sense of Definition \ref{aut-group}.  We believe that adding ``zigzags'' violates this type of convergence.  \end{remark}

In contrast with the existence of Hofer's metric (or the Hofer norm) on the Hamiltonian diffeomorphism group of a symplectic manifold, there does not exist a Finsler-type bi-invariant metric (or, equivalently, a conjugation-invariant norm) on the identity component of contactomorphism group $\Cont_0(M,\xi)$. In fact, any such non-trivial norm will be {\it discrete}, i.e., its values can not be arbitrarily close to zero, see \cite{FPR18, San15}, or the original argument in \cite{BIP06}. To remedy this, we can either restrict ourselves to a smaller group, or on the other hand drop the bi-invariance condition. The first approach is carried out by Banyaga-Donato \cite{BD06}, where a fine non-trivial bi-invariant metric, a modification of the classical Hofer's metric, is defined on the subgroup of strict contactomorphisms, i.e. those contactomorphisms which preserve a contact form;  the second approach is implemented by Shelukhin in \cite{She17}, where he defines a (non-conjugation-invariant) non-degenerate norm on $\Cont_0(M,\xi)$ denoted by $||-||_{\alpha}$ (depending on a prior given contact 1-form $\alpha$ defining $\xi$). In this paper, we follow Shelukhin, and study the pseudo-metric induced by  $||-||_{\alpha}$ on orbit spaces of subsets of $M$, which we refer to as the Shelukhin-Chekanov-Hofer pseudo-metric. The definition and basic properties of the Shelukhin-Hofer norm are recalled in Section \ref{sec-al-norm}.

\begin{dfn} \label{SCH-norm} Let $(M, \xi= \ker\alpha)$ be a contact manifold with a fixed contact 1-form $\alpha$. Fix a subset $N \subset M$, and denote by $\mathcal L(N)$ its orbit space under the action of $\Cont_0(M, \xi)$. Then for any $L_1, L_2 \in \mathcal L(N)$, define 
\begin{equation} \label{dfn-ca}
 \delta_{\alpha}(L_1, L_2) = \inf\{||\phi||_{\alpha} \,| \, \phi \in \Cont_0(M,\xi)\,\,\mbox{s.t.}\,\, \phi(L_1) = L_2 \}.
 \end{equation}
\end{dfn}
For brevity, we will call $\delta_{\alpha}$ the {\it $\alpha$-metric} on $\mathcal L(N)$. Basic properties of this $\alpha$-metric are explored in Section \ref{sec-al-norm}.

Certainly an interesting question is the non-degeneracy of $\delta_{\alpha}$ on $\cL(N)$. Using elementary arguments, in Section \ref{sec-ex-nd} we can show that for several cases $\delta_{\alpha}$ is indeed non-degenerate. Let us point out that different choices of contact forms give rise to equivalent metrics on $\cL(N)$ (see \eqref{eq-SH-norm-equiva} below) and in particular the (non-)degeneracy of $\delta_\alpha$ is independent of this choice.

In general, in order to approach this question, we use Usher's method in \cite{Ush14} based on the following interesting concept called {\it rigid locus} (an analogue to Definition 4.1 in \cite{Ush14}). 

\begin{dfn} \label{dfn-rl} Let $N$ be a subset of a contact manifold $(M, \xi= \ker\alpha)$. The rigid locus of $N$ is defined by 
\begin{equation} \label{rl}
R_N: = \{x \in N\,|\, \mbox{for any $\phi \in \bar{\Sigma}_N$, $\phi(x) \in N$}\}
\end{equation}
where $\Sigma_N$ is the stabilizer of $N$ in $\Cont_0(M, \xi)$ and $\bar{\Sigma}_N$ is its closure with respect to $||-||_{\alpha}$.
\end{dfn}

Then, using a fragmentation type result on $\Cont_0(M,\xi)$, a useful criterion as follows can be obtained and its proof is in given Section \ref{sec-7}.
\begin{prop} \label{rl-c} Let $N$ be a proper closed subset of $M$. Then 
\begin{itemize}
\item[(1)] If $R_N = N$, then $\delta_{\alpha}$ is non-degenerate. 
\item[(2)] If $R_N = \emptyset$, then $\delta_{\alpha} \equiv 0$.
\end{itemize}
\end{prop}

One of the main results in this paper is the following dichotomy, similar to the symplectic case that can be found in Corollary 2.7 in \cite{Ush15}, or Theorem~2 in \cite{Che00}.
\begin{theorem} \label{dic} Let $N$ be a closed connected submanifold of $(M^{2n+1}, \xi = \ker \alpha)$ with $\dim N = n$. Then $\delta_{\alpha}$ is either non-degenerate or vanishes identically. \end{theorem}

The most interesting submanifolds of dimension $n$ in a contact manifold $(M^{2n+1}, \xi)$ are Legendrian submanifolds. At present, it is unclear how $\delta_{\alpha}$ behaves in the orbit space of a Legendrian submanifold. We have the following conjecture. 

\begin{conj} \label{conj-1} Let $N$ be a closed connected Legendrian submanifold of a contact manifold $(M, \xi = \ker \alpha)$, then $\delta_{\alpha}$ is non-degenerate. \end{conj}

In Section \ref{sec-dm}, we briefly point out the difficulty of proving Conjecture~\ref{conj-1}, which roughly speaking comes from the failure to find a uniform subset in the symplectization that can be displaced by all the lifts of contactomorphisms. On the other hand, we claim that if $N$ with dimension $n$ is non-Legendrian (i.e., not Legendrian at some point), then $\delta_{\alpha}$ vanishes identically. In fact, by Proposition \ref{nec}, if $\delta_{\alpha}$ is non-degenerate on $\mathcal L(N)$, then $N$ has to be contact coisotropic. But Legendrian submanifolds are the only contact coisotropic submanifolds of dimension $n$, so the claim follows directly from Theorem \ref{dic}.

\begin{remark} Soon after our work became public, a recent work by Usher \cite{Ush20} confirms Conjecture \ref{conj-1} for a certain family of Legendrian submanifolds where $N$ is hypertight. See Corollary 3.5 in \cite{Ush20} where Theorem 3.3 in \cite{Ush20} provides the key energy estimation. Here, hypertight means that there exists a contact form $\alpha$ defining $\xi$ such that every closed orbit of the Reeb vector field $R_{\alpha}$ is non-contractible and every Reeb chord for $N$ represents a non-trivial element of $\pi_1(M, N)$. \end{remark}

\subsection*{Acknowledgements} We are very grateful for useful communications with Emmanuel Opshtein, Leonid Polterovich, Sobhan Seyfaddini, Egor Shelukhin and Michael Usher. The first author is supported by the German Research Foundation (DFG), CRC/TRR 191 “Symplectic Structures in Geometry, Algebra and Dynamics”. The second author is supported by the European Research Council Advanced grant 338809.

\section{Contact preliminaries and notations}\label{sec-prelim}
Here we recall some basic notion from contact topology and introduce the notations we will use below.

Recall that a co-oriented contact structure on an odd-dimensional manifold $M^{2n+1}$ is a hyperplane distribution $\xi \subset TM$ which can be (globally) defined as a kernel of a $1$-form $\alpha$ such that $\alpha \wedge (d\alpha)^n$ is a volume form on $M$. Any such $1$-form which respects the co-orientation of $\xi$ is called a contact form. A contactomorphism of $(M, \alpha)$ is a diffeomorphism $\phi : M \to M$ which preserves $\xi$ together with its co-orientation, which in terms of a contact form $\alpha$ reads $\phi^*\alpha = e^g \alpha$ or some smooth function $g$, called the conformal factor of $\phi$ (with respect to $\alpha$). We denote the group of contactomorphisms of $(M, \xi)$ by $\Cont(M, \xi)$, and its identity component by ${\rm Cont}_0(M, \xi)$. 

The Reeb vector field associated with a contact form $\alpha$ is the unique vector field $R_\alpha$ on $M$ satisfying 
\begin{equation*}
\alpha(R_\alpha)=1 \quad \text{and} \quad i_{R_\alpha}d\alpha=0.
\end{equation*}

A smooth time-dependent function $F_t : M \to \R$, $t \in [0,1]$, called in this context a contact Hamiltonian, determines a unique vector field $X_{F_t}$ on $M$ by the requirements
\begin{equation}\label{CVF2}
\alpha(X_{F_t})=F_t \quad \text{and} \quad i_{X_{F_t}}d\alpha = dF_t(R_\alpha) \alpha - dF_t.
\end{equation}
If follows that $\cL_{X_{F_t}}\alpha$ is proportional to $\alpha$, and hence the flow of $X_{F_t}$ gives rise to a contact isotopy $\phi_t$. We say that $\{\phi_t\}$ is generated by the contact Hamiltonian $F_t$, and note that any contact isotopy is generated by a (unique) contact Hamiltonian. As an example, the Reeb flow is generated by the constant Hamiltonian $F_t \equiv 1$.

The symplectization $SM$ of $(M,\xi=\ker\alpha)$ is the symplectic manifold $SM=(\R \times M, d(e^\theta \alpha))$, where $\theta$ is the coordinate on $\R$. Let us point out that $SM$ is, up to a symplectomophsim, independent of the choice of $\alpha$. Any contactomorphism $\phi$ of $M$ lifts to a symplectomorphism $\tilde{\phi}$ of $SM$, defined by $$\tilde{\phi}(\theta, m)=(\theta-g(m), \phi(m)),$$ where $g$ is the conformal factor of $\phi$.
If  $\phi_t$  is a contact isotopy of $M$ generated by a contact Hamiltonian $F_t$, the lifted isotopy $\tilde{\phi}_t$ is a Hamiltonian isotopy of $SM$ generated by the Hamiltonian $\tilde{F}_t(\theta,m)=e^\theta F_t(m)$. The Hamiltonian vector field $X_{\tilde{F}_t }$ is related to the contact vector field $X_{F_t}$ by
\begin{equation}\label{eq-lift-vf}
X_{\tilde{F}_t }=-dF_t(R_\alpha) \frac{\partial}{\partial \theta}+X_{F_t} \in T\R \oplus TM \simeq T(SM).
\end{equation}

An central role in our story is played by the contact Poisson bracket. This is a Lie bracket on $C^\infty(M)$ defined by 
\begin{equation}\label{CPB}
\{F,G\}_{\alpha}=dF(X_G)-dG(R_\alpha)F.
\end{equation}
It is related to the Poisson bracket on the symplectic manifold $SM$ via (see Exercise 3.57 (iv) in \cite{MS95})
\begin{equation}\label{eq-pb-correspond}
\{\tilde{F},\tilde{G}\}=e^\theta \{F,G\}_{\alpha}.
\end{equation}
Let us emphasize that $\{-,-\}_{\alpha}$ depends on $\alpha$. Interested readers can work out the formula of $\{-,-\}_{\alpha'}$ in terms of $\{-,-\}_{\alpha}$ if $\alpha' = e^f \alpha$ for some $f \in C^{\infty}(M)$. 

\section{Contact coisotropic submanifold}
 
\subsection{Proof of Proposition \ref{C-P}} 

Given a smooth function $G:M \to \R$, we denote by $\phi_G^t$ the contact isotopy generated by $G$.

\begin{proof} [Proof of Proposition \ref{C-P}] Suppose $Y$ is coisotropic. Let $F|_Y = G|_Y = 0$. For any $p \in Y$, $\alpha_p(X_G(p)) = G(p) = 0$, so $X_G(p) \subset \xi_p$. For every $v \in T_pY \cap \xi_p$, by (\ref{CVF2}),
\[ (d\alpha)_p(X_G(p), v) = dG_p(R_{\alpha}) \alpha(v)- dG_p(v) = - dG_p(v) = 0 \]
where the final equality comes from our assumption that $G$ is constant along $Y$. So $X_G(p) \subset (T_pY \cap \xi_p)^{\perp_{d\alpha}} \subset T_pY \cap \xi_p \subset T_p Y$. Now by definition (\ref{CPB}), for any $p \in Y$, 
\begin{align*}
\{F,G\}_{\alpha}(p) & = \frac{d}{dt} F(\phi_G^t(p))|_{t=0} - dG_p(R_{\alpha}) \cdot F(p) \\
& = \frac{d}{dt} F(\phi_G^t(p))|_{t=0} - 0 = 0
\end{align*}
where the final equality comes from our assumption that $F$ is constant along $Y$. Thus we have proved $\{F, G\}_{\alpha}|_{Y} = 0$. 

Conversely, suppose $I_Y$ is a Lie subalgebra. If there exists some $p \in Y$ and $v \in (T_{p} Y \cap \xi_{p})^{\perp_{d\alpha}}$ but $v \notin T_{p} Y \cap \xi_{p}$, then one can find a function $G \in I_Y$ such that $(dG)_{p}(v) \neq 0$. Then by equation (\ref{CVF2}),
\[(d\alpha)_{p}(X_{G}(p), v) = -(dG)_{p}(v) \neq 0\]
which implies $X_{G}(p) \notin ((T_{p} Y \cap \xi_{p})^{\perp_{d\alpha}})^{\perp_{d\alpha}} = T_{p} Y \cap \xi_{p}$. Again, one can find a function $H \in I_Y$ such that $(dH)_{p}(X_{G}(p)) \neq 0$. Hence
\[ \{H, G\}_{\alpha}(p) = (dH)_{p}(X_{G}(p)) \neq 0,\]
and we get a contradiction. 
\end{proof}

A direct consequence of Proposition \ref{C-P} is

\begin{cor} \label{cont-coiso} Let $(M, \xi = \ker\alpha)$ be a contact manifold. Suppose $Y \subset M$ is a contact coisotropic submanifold and $\phi$ is a contactomorphism on $M$. Then $\phi(Y)$ is also contact coisotropic. \end{cor}

This corollary easily follows from the following lemma which is of interest itself. 

\begin{lemma} \label{res-CVF} Suppose $\phi$ is a contactomorphism on $(M, \xi = \ker\alpha)$ with conformal factor $g:M \to \R$. For any functions $F, G \in C_c^{\infty}(M)$, 
\[   \{e^{-g} \phi^*F, e^{-g} \phi^*G\}_{\alpha} = e^{-g} \phi^*\{F, G\}_{\alpha}.\]
\end{lemma}

Assuming this lemma, we have

\begin{proof} [Proof of Corollary \ref{cont-coiso}] Let $g$ be the conformal factor of contactomorphism $\phi$. For any two functions $F, G$ such that $F|_{\phi(Y)} = G|_{\phi(Y)} = 0$, we have $(e^{-g} \phi^* F)|_Y = (e^{-g} \phi^* G)|_Y = 0$. Then since $Y$ is coisotropic, by Proposition~\ref{C-P} and Lemma \ref{res-CVF}, 
\[ \{e^{-g} \phi^*F, e^{-g} \phi^* G\}_{\alpha} |_Y = 0 = e^{-g} \phi^*\{F, G\}_{\alpha}|_Y. \]
Hence, $\{F, G\}_{\alpha}|_{\phi(Y)} = 0$. Then by Proposition \ref{C-P} again, $\phi(Y)$ is coisotropic. \end{proof}

Now let us prove Lemma \ref{res-CVF}. 

\begin{proof}[Proof of Lemma \ref{res-CVF}] First, note that the lifts $\tilde{F}$ and $\tilde{\phi}$ of $F$ and $\phi$, respectively, to $SM$ (see Section \ref{sec-prelim}) satisfy
\[ (\tilde{\phi})^* \tilde{F} = \tilde{F'} \]
where $F': M \to \R$ is defined by $F'(m) = e^{-g(m)} F(\phi(m))$. 
Then by symplectic invariance property of the (symplectic) Poisson bracket on the symplectization and relation \eqref{eq-pb-correspond}, we compute 
\begin{align*}
    e^{\theta}\{F', G'\}_{\alpha} & = \{\tilde{F}', \tilde{G}'\}\\
    & = \{(\tilde{\phi})^*\tilde{F},(\tilde{\phi})^*\tilde{G}\}\\
    & = (\tilde{\phi})^*\{\tilde{F}, \tilde{G}\}\\
    & = (\tilde{\phi})^*(e^{\theta}\{F, G\}_{\alpha})\\
    & = e^{\theta-g}\phi^*\{F, G\}_{\alpha}.
\end{align*}
That is $e^{\theta} \{F', G'\}_{\alpha} = e^{\theta - g}\phi^*\{F, G\}_{\alpha}$. Therefore, $\{F',G'\}_{\alpha} = e^{-g} \phi^*\{F, G\}_{\alpha}$ which is the desired conclusion. \end{proof}

\subsection{Examples} 

\begin{ex} Suppose $Y \subset M^{2n+1}$ is a Legendrian submanifold. For any $p \in Y$, $T_pY \cap \xi_p = T_p Y$. Moreover, since $T_pY$ is a Lagrangian subspace with respect to $d\alpha|_{\xi}$, $(T_pY \cap \xi_p)^{\perp_{d\alpha}} = (T_p Y)^{\perp_{d\alpha}} = T_p Y = T_pY \cap \xi_p$. So in particular, $Y$ is contact coisotropic. Note that by Definition \ref{dfn-Hua15}, $\dim(Y) \geq n$, so Legendrian submanifolds provide the lowest dimension within all the coisotropic submanifolds. In fact, a coisotropic submanifold of dimension $n ( = \frac{1}{2} (\dim M - 1))$ is indeed a Legendrian submanifold. \end{ex}

\begin{ex} \label{hyper} Let $\Sigma$ be a hypersurface of a contact manifold $(M, \xi = \ker \alpha)$. Then $\Sigma$ is contact coisotropic. In fact, for any point $x \in \Sigma$, there are two cases. One is $T_x \Sigma = \xi_x$, then $(T_x \Sigma \cap \xi_x)^{\perp_{d\alpha}} = \{0\} \subset T_x \Sigma \cap \xi_x$. The other is $T_x \Sigma$ is transversal to $\xi_x$, then by dimension counting, $T_x \Sigma \cap \xi_x$ is a hyperplane of $\xi_x$, therefore, symplectic coisotropic with respect to $d\alpha|_{\xi}$. Thus we get the conclusion. \end{ex}

\begin{ex} \label{Pre-Lagrangian} Recall that a submanifold $L \subset M^{2n+1}$ of dimension $n+1$ is called a {\it pre-Lagrangian} if it admits a Lagrangian lift in the symplectization of $M$, that is, a Lagrangian submanifold $\hat{L} \subset SM$ such that the canonical projection $\pi: SM \to M$ restricts to a diffeomorphism $\hat{L} \to L$. We claim any pre-Lagrangian submanifold is contact coisotropic. In fact, let $L \subset M$ be a pre-Lagrangian and $F, G \in C_c^{\infty}(M)$ such that $F|_{L} = G|_{L} =0$. Then their lifts $\tilde{F}, \tilde{G}$ to $SM$ satisfy $\tilde{F}|_{\hat{L}}=\tilde{G}|_{\hat{L}}=0$, and hence by \eqref{eq-pb-correspond}, 
\[ 0 = \{\tilde{F}, \tilde{G}\}|_{\hat{L}}  = e^{\theta} \cdot \{F, G\}_{\alpha}|_{L} \]
which implies $\{F, G\}_{\alpha} =0$ on $L$. By Proposition \ref{C-P}, $L$ is contact coisotropic. \end{ex}

To end this section, we want to address the following point. Usually Legendrian submanifolds are viewed as analogues to Lagrangian submanifolds in symplectic topology. However, we use the following concept to illustrate that, to some extent, pre-Lagrangian submanifolds are closer to Lagrangian submanifolds than Legendrian submanifolds.  

\begin{dfn} \label{dfn-inf-dis}
A submanifold $N$ of a contact manifold $(M, \xi)$ is called {\it infinitesimally displaceable} if there exists a function $H: M \to \R$ such that for any point $x \in N$, the contact vector field $X_H(x) \notin T_xN$.  \end{dfn}

\begin{ex} \label{ex-inf-dis} Any compact Legendrian submanifold is infinitesimally displaceable simply by flowing along a Reeb vector field for a sufficiently small amount of time. This is in sharp contrast to the well-known fact that any Lagrangian submanifold is  {\it not} infinitesimally displaceable (cf. Proposition 4.8 in \cite{Ush14}). \end{ex}

\begin{lemma} Any compact pre-Lagrangian submanifold is not infinitesimally displaceable. \end{lemma}

\begin{proof} 
Suppose $N$ is a compact pre-Lagrangian of $M$, then there exists a compact Lagrangian submanifold $L \subset SM$ such that $\pi(L) = N$ where $\pi:  SM \to M$ is the canonical projection. Moreover, for any Hamiltonian function $H$ on $M$ and its lift $\tilde{H}$ to $SM$, the corresponding contact and Hamiltonian vector field satisfy, in view of \eqref{eq-lift-vf}, that $\pi_*(X_{\tilde H}) = X_H$. If $H$ is a function such that $X_H(p) \notin T_pN$ for any $p \in N$, then $\tilde{H}$ is a function such that $X_{\tilde{H}}(x) \notin T_x L$ for any $x \in L$. This implies that this (compact) Lagrangian submanifold $L$ is infinitesimally displaceable, which is a contradiction.  \end{proof}

\begin{remark} In fact, in the contact topology set-up, Legendrian submanifolds are not the only coisotropic submanifolds that break the infinitesimally displaceable rule. In a 3-dimensional contact manifold $(M, \xi)$, there exists a special class surfaces, called {\it convex} surface, introduced by Giroux in \cite{Gir91}, which admits a transversal contact vector field. Hence, this provides a family of contact coisotropic submanifolds, not Legendrian, which are also infinitesimally displaceable. In particular, any such convex surface can never be a pre-Lagrangian. It is one of Giroux's great achievements that a $C^{\infty}$-generic closed embedded surface is convex (see Proposition II 2.6 in \cite{Gir91}), therefore any pre-Lagrangian submanifold  of a contact $3$-manifold can be $C^{\infty}$-perturbed into a convex surface, which shows that infinitesimal non-displaceability is not stable at all.\end{remark}

\section{Coisotropic correspondence} \label{sec-cor}
\subsection{Symplectization} 
Let $L$ be a Legendrian submanifold of a contact manifold $(M, \xi = \ker \alpha)$. By Exercise 3.57 (i) in \cite{MS95}, its lift $\R \times L$ is a Lagrangian submanifold in the symplectization $SM=(\R \times M, d(e^{\theta} \alpha))$. Note that both Legendrian submanifolds and Lagrangian submanifolds are special cases of coisotropic submanifolds in the contact and symplectic topology, respectively. Here we have the following generalized result. 

\begin{prop} \label{corres} Let $Y$ be a submanifold of a contact manifold $(M, \xi = \ker \alpha)$. Then $Y$ is a contact coisotropic submanifold of $(M, \xi = \ker \alpha)$ if and only if $\R \times Y$ is a symplectic coisotropic submanifold of the symplectization $SM=(\R \times M, \omega = d(e^{\theta} \alpha))$. 
\end{prop}

\begin{proof} [Proof of Proposition \ref{corres}]
We claim that, for any submanifold $Y \subset M$ and any $p \in Y$ and $\theta \in \R$,
\begin{equation}\label{eq-orthogonal}
\pi_*((T_\theta\R \times T_p Y)^{\perp_\omega}) = (T_p Y \cap \xi_p)^{\perp_{d\alpha}},
\end{equation}
where $\pi \colon SM \to M$ is the canonical projection. Assuming \eqref{eq-orthogonal}, we note that if $Y \subset M$ is contact coisotropic, then
$$
\pi_*((T_\theta\R \times T_p Y)^{\perp_\omega}) = (T_p Y \cap \xi_p)^{\perp_{d\alpha}} \subset T_p Y,
$$
and hence
$$
(T_\theta\R \times T_p Y)^{\perp_\Omega} \subset T_\theta\R \times T_p Y = T_{(\theta, p)} (\R \times Y),
$$
proving that $\R \times Y$ is a symplectic coisotropic submanifold of $SM$. Conversely, if $\R \times Y \subset SM$ is symplectic coisotropic, we get that
$$
(T_pY \cap \xi_p)^{\perp_{\alpha}} = \pi_*((T_\theta\R \times T_p Y)^{\perp_\omega}) \subset \pi_*(T_\theta\R \times T_p Y) = T_pY,
$$
proving that $Y$ is a contact coisotropic submanifold. It remains to prove \eqref{eq-orthogonal}. Note that $v \in (T_p Y \cap \xi_p)^{\perp_{d\alpha }}$ if 
\begin{equation}\label{eq-perp_alpha}
\alpha(v)=0 \quad \text{and}\quad \forall u\in T_pY,\, \alpha(u)=0\Rightarrow d\alpha(u,v)=0.
\end{equation}
On the other hand, $(s,v) \in (T_\theta\R \times T_p Y)^{\perp_\omega}$ if
$$
\forall u\in T_pY, \, t\in T_\theta\R,\,\,\, \omega((t,u),(s,v)) = e^\theta(t\alpha(v)-s\alpha(u)+d\alpha(u,v))=0.
$$

Note that the latter condition holds for all $t\in T_\theta\R$, which implies $\alpha(v)=0$, so we can rewrite this condition as
\begin{equation}\label{eq-perp-Omega}
\alpha(v)=0  \quad \text{and}\quad \forall u \in T_pY, \,d\alpha(u,v) = s\alpha(u).
\end{equation}
We now consider separately two cases for the point $p\in Y$: 
\begin{itemize}
    \item [I.] If $T_pY \subset \xi_p$, that is $\alpha(u)=0$ for all $u \in T_pY$, we note that conditions \eqref{eq-perp_alpha} and \eqref{eq-perp-Omega} both reduce to $d\alpha(u,v)=0$ for all $u \in T_pY$ and hence 
$$(T_\theta\R \times T_pY)^{\perp_{\omega}} = T_\theta\R \times (T_pY \cap \xi_p)^{\perp_{d\alpha}}.$$
\item[II.] If $T_pY \not\subset \xi_p$, then $\dim(T_pY / (T_p Y \cap \xi_p))=1$. Note that if $v$ satisfies \eqref{eq-perp_alpha}, the function $d\alpha (\cdot, v)$ descends to $T_pY / (T_pY \cap \xi_p)$, and therefore the number $s(v) = \frac{d\alpha(u,v)}{\alpha(u)}$ is independent of $u \in T_pY$ such that $\alpha(u) \neq 0$. It follows easily that 
$$(T_\theta\R \times T_pY)^{\perp_{\omega}} = \{(s(v), v) \,|\, v \in (T_pY \cap \xi_p)^{\perp_{d\alpha}} \} .$$
\end{itemize}
In both cases, we see that \eqref{eq-orthogonal} holds, which completes the proof.
\end{proof}

\subsection{Prequantization} Recall that if a closed symplectic manifold $(M, \omega)$ satisfies the condition $[\omega] \in H^2(M; \Z)$, then there exists a principal $S^1$-bundle $P \xrightarrow{\pi} M$ such that $P$ has a contact structure $(P, \xi= \ker \alpha)$ for some $\alpha$ satisfying $d\alpha = \pi^*\omega$. This contact manifold $(P, \xi)$ is called a {\it prequantization of $(M, \omega)$}. It is easy to check that for any point $a \in P$, we have a decomposition $T_aP = \xi_a \oplus V_a$ where $V_a$ is the vertical subspace corresponding to the $S^1$-direction (viewing $\alpha$ as a connection form). Moreover, via $d\pi$, we can identify 
\[ \xi_a \simeq T_x M \,\,\,\,\mbox{for any $a \in \pi^{-1}(x)$}.\]
The relation between contact Hamiltonian vector fields and symplectic Hamiltonian vector fields is expressed by the following equation. For any Hamiltonian function $F$ on $M $, and any point $a \in P$,  denoting $x = \pi(a)$, one has 
\begin{equation} \label{cs-vf}
X_{\pi^*F}(a) = \pi_*^{-1}(X_F(x)) \oplus v 
\end{equation}
where $v$ is the unique vector in $V_a$ such that $\alpha(v) = F(x)$ (cf. Section 1.3 in \cite{BS10}). 

\begin{prop} \label{pre-cor}
Let $\Lambda \subset M$ be a submanifold. Then $\Lambda$ is symplectic coisotropic submanifold of $M$ if and only if $\pi^{-1}(\Lambda)$ is a contact coisotropic submanifold of $P$. 
\end{prop}

\begin{proof} For $a \in \pi^{-1}(\Lambda)$, via $d\pi$, we can identify 
\[ T_a(\pi^{-1}(\Lambda)) \cap \xi_a \simeq (T_{\pi(a)} \Lambda \oplus V_a) \cap T_xM= T_{\pi(a)} \Lambda\]
and $d\alpha$ is identified with $\omega$. 

\medskip
If $\Lambda$ is symplectic coisotropic, then 
\[ (T_a(\pi^{-1}\Lambda) \cap \xi_a)^{\perp_{d\alpha}} =(T_{\pi(a)} \Lambda)^{\perp_{\omega}} \subset T_{\pi(a)}\Lambda \simeq T_a(\pi^{-1}(\Lambda)) \cap \xi_a. \]
Therefore, $\pi^{-1}(\Lambda)$ is contact coisotropic. The same relation above also shows that if $\pi^{-1}(\Lambda)$ is contact coisotropic, then $\Lambda$ is symplectic coisotropic. 
\end{proof}

\begin{remark} Part of the proofs of Proposition \ref{corres} and Proposition \ref{pre-cor} can be done via the criterion from symplectic and contact Poisson brackets. For instance, the equation (\ref{eq-pb-correspond}) above easily implies that if $\R \times Y$ is symplectic coisotropic in the symplectization $SM$, then $Y$ is contact coisotropic in the base contact manifold $M$. For another instance, one can show that, for any two functions $F, G \in C^{\infty}(M)$, their lifts to the prequantization $\pi^*F, \pi^*G$ satisfy 
\[ \{\pi^*F, \pi^*G\}_{\alpha} = \pi^*\{F, G\}. \]
This is analogous to the equation (\ref{eq-pb-correspond}) above, and it implies that if $\pi^{-1}(\Lambda)$ is contact coisotropic in the prequantization $P$, then $\Lambda$ is a symplectic coisotropic in the base symplectic manifold $M$. 
\end{remark}

\begin{remark} Since the  pre-image $\pi^{-1}(\Lambda)$ always contains the full $S^1$-fibres, we know that $\dim \pi^{-1}(\Lambda) = \dim \Lambda +1$. Moreover, since the $S^1$-direction is always transversal to the contact hyperplane $\xi$ of $P$, if $\Lambda$ is symplectic coisotropic, then $\pi^{-1} (\Lambda)$ is always a strictly coisotropic submanifold with respect to the contact 1-form $\alpha$ on $P$ (see Definition 1.6 in \cite{BZ15}). A detailed study of these submanifolds was carried out in \cite{BZ15}. 
\end{remark}

\section{Proof of Theorem \ref{rigidity}}
Recall that a symplectic homeomorphism $\phi$ is a homeomorphism which is a $C^0$-limit of symplectomorphisms. The collection of all symplectic homeomorphisms of a symplectic manifold $(M, \omega)$ is denoted as $\Sympeo(M, \omega)$. The famous Gromov-Eliashberg theorem (see \cite{Eli87}) states that 
\begin{equation} \label{GE}
\Sympeo(M, \omega) \cap \Diff(M) = \Symp(M, \omega).
\end{equation}
In the contact topology set-up, one can define 

\begin{dfn} \label{aut-group} (Definition 6.8 in \cite{MS15}) A homeomorphism $\phi$ of a contact manifold $(M, \xi = \ker \alpha)$ is called {\it a topological automorphism of the contact structure $\xi$} if there exists a sequence of contactomorphisms $\phi_k \in \Cont(M, \xi)$ with corresponding conformal factors $g_k: M \to \R$ such that 
\begin{itemize}
\item{} $\phi_k \xrightarrow{C^0} \phi$;
\item{} $g_k \xrightarrow{C^0} g$ for some continuous function $g: M \to \R$.
\end{itemize}
Here the $C^0$-topology on the group of homeomorphisms of $M$ is induced by the metric $d_{C^0}(\phi, \psi) = \max_{M} d_M(\phi(x), \psi(x))$ where $d_M$ is a distance function defined with respect to a fixed Riemannian metric on $M$. The collection of all such homeomorphisms is denoted by $\Aut(M, \xi)$. 
\end{dfn}
Then an analogue conclusion to (\ref{GE}) is Theorem 1.3 in \cite{MS15} which states that 
\begin{equation} \label{MS}
\Aut(M, \xi) \cap \Diff(M) = \Cont(M, \xi).
\end{equation}

Conclusions (\ref{GE}) and (\ref{MS}) can be implied by their ``relative'' counterparts, i.e. conclusions involving Lagrangian submanifolds and Legendrian submanifolds, respectively. 

\begin{ex} \label{ex-GE} Let $\phi \in \Symp(M, \omega)$. Consider the embedding $\Phi_{\phi}: M \to M \times M$ given by $x \to (x, \phi(x))$. Under the symplectic structure $\pi^*_1 \omega \oplus \pi_2^*(-\omega)$ on $M \times M$, ${\rm Im}(\Phi_{\phi})$ is a Lagrangian submanifold. In fact, for $\phi \in \Diff(M)$, ${\rm Im}(\Phi_{\phi})$ is Lagrangian if and only if $\phi$ is symplectic. Then the conclusion (\ref{GE}) is implied by the following statement \cite{LS94}: if $\phi \in \Sympeo(M, \omega)$ and ${\rm \Phi_{\phi}}$ is smooth, then ${\rm \Phi_{\phi}}$ is Lagrangian. Indeed, for any $\phi \in \Sympeo(M, \omega) \cap \Diff(M)$, $\Phi_{\phi}$ is smooth. Then ${\rm \Phi_{\phi}}$ is Lagrangian, which implies $\phi \in \Symp(M, \omega)$. Hence, $\Sympeo(M, \omega) \cap \Diff(M) \subset \Symp(M, \omega)$, and $ \Symp(M, \omega) \subset \Sympeo(M, \omega) \cap \Diff(M)$ is trivial. 
\end{ex}

\begin{ex} \label{ex-MS} Let $\phi \in \Cont(M, \xi)$ with conformal factor $g$. Consider the embedding $\Psi_{\phi, g}: M \to M \times M \times \R$ given by $x \to (x, \phi(x), g(x))$. Then fixing a contact 1-form $\alpha$ on $M$, under the contact 1-form $e^{\theta} \pi_1^* \alpha \oplus \pi_2^*(-\alpha)$ on $M \times M \times \R$ (where $\theta$ is the $\R$-coordinate), ${\rm Im}(\Psi_{\phi, g})$ is a Legendrian submanifold. In fact, it is easy to check that for $\phi \in \Diff(M)$ and a smooth function $g: M \to \R$, the image ${\rm Im}(\Psi_{\phi, g})$ is Legendrian if and only if $\phi \in \Cont(M, \xi)$ with conformal factor $g$. Then, similarly to Example \ref{ex-GE}, the conclusion (\ref{MS}) is implied by the following statement (a special case of Theorem \ref{rigidity}): if $\phi \in \Aut(M, \xi)$ with limiting conformal factor $g$ and ${\rm Im}(\Psi_{\phi, g})$ is smooth, then ${\rm Im}(\Psi_{\phi, g})$ is Legendrian. \end{ex}

A recent work from \cite{HLS15} generalizes the observation in Example \ref{ex-GE} from Lagrangian submanifolds to symplectic coisotropic submanifolds. Roughly speaking, the main result proves that symplectic homeomorphisms preserve symplectic coisotropic submanifolds. Similarly, our Theorem \ref{rigidity} is a natural generalization of Example \ref{ex-MS}. 

\begin{proof} [Proof of Theorem \ref{rigidity}] By Proposition \ref{corres}, $\R \times Y \subset SM$ is a (symplectic) coisotropic submanifold. Suppose there exists a $\phi_k \xrightarrow{C^0} \phi$ where $\phi_k$ are contactomorphisms with conformal factors $g_k \xrightarrow{C^0} g$ for some continuous function $g: M \to \R$. Recall from Section \ref{sec-prelim} that $\phi_k$ lifts to a symplectomorphism
\[ \Phi_k : SM \to SM, \quad  \Phi_k(\theta,m) = (\theta- g_k(m), \phi_k(m)). \]
Moreover, by defining $\Phi(\theta,m) = (\theta - g(m), \phi(m))$ (as the formal lift of $\phi$), $\Phi_k \xrightarrow{C^0} \Phi$ and thus $\Phi \in {\Sympeo(SM)}$. Since $\Phi(\R \times Y) = \R \times \phi(Y)$, by our assumption on the smoothness of $\phi(Y)$, $\Phi(\R \times Y)$ is smooth. The the main theorem in \cite{HLS15} implies that $\Phi(\R \times Y) = \R \times \phi(Y)$ is symplectic coisotropic. Finally, by (ii) in Proposition \ref{corres}, $\phi(Y)$ is contact coisotropic. \end{proof}

\begin{remark} The $C^0$-rigidity of contactomorphisms also holds when we drop the requirement of convergence of conformal factors, see \cite{MS14}. In this case, since we have no control on the behaviour of the limiting conformal factor, back in the proof of Theorem \ref{rigidity}, there is no guarantee that the lift $\Phi_k$ is $C^0$-convergent to $\Phi$. As a matter of fact, we highly doubt the same conclusion in Theorem \ref{rigidity} still holds in this case. \end{remark}

\section{Shelukhin-Chekanov-Hofer's metric} \label{sec-al-norm}
Let $(M, \xi = \ker\alpha)$ be a contact manifold with a fixed contact 1-form $\alpha$. For any $\phi \in \Cont_0(M, \xi)$, \cite{She17} defines the following norm 
\begin{equation} \label{norm-she}
||\phi||_{\alpha} : = \inf_{\phi_{H}^1 = \phi} \int_0^1 \max_M |H(\cdot, t)| dt.
\end{equation}
where the infimum is taken over all the contact isotopies with time-1 map being $\phi$ and $H \in C^{\infty}(M \times [0,1])$ denotes the associated contact Hamiltonian functions. Let us recall some properties of $||-||_\alpha$ proven in \cite{She17}. Though {\it not} conjugation invariant, it satisfies a natural ``coordinate-change'' formula as follows, for any $\phi, \psi \in \Cont_0(M, \xi)$, 
\begin{equation} \label{conj}
||\psi \phi \psi^{-1}||_{\alpha} = ||\phi||_{\psi^* \alpha}.
\end{equation}
Additionally, different choice of contact forms give rise to equivalent norms. Specifically, if $\beta = f\alpha$ is a another contact form, where $f : M \to (0,\infty)$ is a smooth function, then (see \cite[Lemma 10]{She17})
\begin{equation}\label{eq-SH-norm-equiva}
\min_M f \cdot ||-||_\alpha \leq ||-||_\beta \leq \max_M f \cdot ||-||_\alpha.
\end{equation}
In particular, one has
\begin{equation}\label{eq-quasi-conj-inv}
 C_-(\psi) ||\phi||_\alpha\leq||\psi \phi \psi^{-1}||_{\alpha} \leq C_+(\psi) ||\phi||_\alpha,
\end{equation}
where $C_+(\psi) = \max_M e^g$ for conformal factor $\psi^*\alpha = e^g \alpha$, and $C_-(\psi)= C_+(\psi^{-1})^{-1}$. Finally, the $\alpha$-displacement energy of a subset $U \subset M$ is defined by 
\[E_\alpha(U) = \inf \{||\phi||_\alpha \,:\, \psi(U) \cap \overline{U} = \emptyset  \}.
\]
Then (see \cite[Proposition 15]{She17})
\begin{equation}\label{eq-disp-positive}
E_\alpha(U) > 0 \text{ for any open $U\subset M$}.
\end{equation}

Recall the definition of the Shelukhin-Chekanov-Hofer's metric (or for brevity, $\alpha$-metric) (Definition \ref{SCH-norm}): We fix a subset $N 
\subset M$, and denote by $\mathcal{L}(N)$ its orbit under the action of $\Cont_0(M)$. Then
$$
\delta_\alpha (L_1, L_2) = \inf\{||\phi||_\alpha \,|\, \phi(L_1)=L_2  \}.
$$
We can prove the following proposition. 

\begin{prop} \label{prop-a} 
The $\alpha$-metric satisfies the following properties.
\begin{itemize}
\item[(1)] $\delta_{\alpha}(L, L) =0$.
\item[(2)] $\delta_{\alpha}(L_1, L_2) = \delta_{\alpha}(L_2, L_1)$.
\item[(3)] $\delta_{\alpha}(L_1, L_3) \leq \delta_{\alpha}(L_1, L_2) + \delta_{\alpha}(L_2, L_3)$.
\item[(4)] Let $\psi \in \Cont(M, \xi)$, and let $C_\pm(\psi)$ as above. Then
\begin{equation} \label{(4)}
C_-(\psi)\cdot \delta_{\alpha}(L_1, L_2) \leq \delta_{\alpha}(\psi(L_1), \psi(L_2)) \leq C_+(\psi)\cdot \delta_{\alpha}(L_1, L_2).
\end{equation}
\end{itemize}
\end{prop}

\begin{proof} [Proof of Proposition \ref{prop-a}]
The only property which is not automatic is item (4). This follows from the fact that $\phi$ satisfies $\phi(L_1) = L_2$ if and only if $(\psi\phi\psi^{-1}) (\psi(L_1))=\psi(L_2)$, together with \eqref{eq-quasi-conj-inv}.
\end{proof}

\section{Examples of non-degenerate $\delta_{\alpha}$} \label{sec-ex-nd}

In this section, we give some examples of submanifolds $N \subset M$ such that $\delta_{\alpha}$ is non-degenerate on $\mathcal L(N)$. Before this, we want to address a useful reformulation of $\delta_{\alpha}$. By Proposition 1.2 in \cite{Ush15}, 
\begin{equation}\label{eq-delta-Usher}
\delta_\alpha (L_1, L_2) = \inf\left\{ \int_0^1 \max_{\phi_H^t(L_1)}|H_t| dt \,\,\bigg|\,\, \phi_H^1(L_1) = L_2  \right\}.
\end{equation}
\subsection{The simplest example - $(S^1, ds)$}

Consider the case when $M = S^1$ (with co-ordinate $s$ mod $1$), $\alpha = ds$ and $N = \{p\}$, for some $p \in S^1$. $N$ is Legendrian, so contact coisotropic. In this case $\mathcal L(N) = S^1$. Denote by $d_{S^1}$ the standard (angular) distance function on $S^1 = \R / \Z$.
\begin{prop}
In this case $\delta_\alpha = d_{S^1}$.
\end{prop}
\begin{proof}
Indeed, let $p\neq q \in S^1$. In our case, formula \eqref{eq-delta-Usher} simplifies to
$$
\delta_\alpha (p,q) = \inf \left\{\int_0^1 |H_t(\phi_H^t (p))|dt \,\,\bigg|\,\, \phi_H^1(p)=q \right\}.
$$
So, let $H_t$ be any contact Hamiltonian generating a contact isotopy $\phi_H^t$ with $\phi_H^1(p)=q$. Denote $\gamma(t) = \phi_H^t(p)$. Then $\gamma'(t) = X_{H_t}(\phi_H^t(p))$, and hence $\alpha(\gamma'(t)) = H_t(\phi_H^t(p))$. Note, moreover, that with respect to the standard Riammanian metric on $S^1$ (which induces the metric $d_{S^1}$) $|\gamma'(t)| = |\alpha(\gamma'(t))|$, and therefore 
$$
d_{S^1}(p,q) = d_{S^1}(\gamma(0), \gamma(1)) \leq \int_{0}^{1} |\gamma'(t)|dt = \int_0^1 |H_t(\phi_{H}^t(p))|dt.
$$
This proves that $d_{S^1}(p,q) \leq \delta_\alpha (p, q)$. The converse inequality follows easily by considering the isotopy $\{s \mapsto s \pm t\}_{0 \leq  t \leq d_{S^1}(p,q)}$ (the sign chosen positive if $q$ is obtained from $p$ via a counter-clockwise rotation by $d_{S^1}(p,q)$, and negative otherwise).
\end{proof}

\subsection{Pre-Lagrangian case} \label{ss-pl}

\begin{prop}
Let $N$ be any compact manifold and $M = T^*N \times S^1$ with the canonical contact structure induced by the 1-form $\alpha = \lambda_{\rm can} + dt$. Then for $L_0 = o_N \times S^1$, we have that $\delta_\alpha$ is non-degenerate on $\mathcal L(L_0)$.
\end{prop}

\begin{proof}
First, observe that there exists a standard symplectomorphism between $SM$ and $T^*N \times T^*S^1$ (with the split symplectic structure), and under this symplectomorphism $L_0$ is a pre-Lagrangian submanifold of $M$ with a Lagrangian lift $\hat{L}_0 = o_N \times (S^1 \times \{1\})$ in $T^*N \times T^*S^1$.  Moreover, any $L \in \mathcal L(L_{0})$ is pre-Lagrangian. Indeed, if $L = \phi(L_0) \in \mathcal L(L_0)$, then $\tilde{\phi}(\hat{L}_0)$ is a Lagrangian lift of $L_{0}$ where $\tilde{\phi}$ is the lifted Hamiltonian diffeomorphism. Now, let $L_1, L_2 \in \mathcal L(L_0)$. We claim  that there exists $C > 0$ such that 
\begin{equation}\label{eq-lower-bd}
\delta_{CH}(\hat{L}_1, \hat{L}_2) \leq C \cdot \delta_\alpha(L_1, L_2)
\end{equation}
where $\delta_{CH}$ denotes the Chekanov-Hofer (pseudo)-metric between the Lagrangian submanifolds $\hat{L}_1$ and $\hat{L}_2$. Since $SM \simeq T^*(N \times S^1)$ which is in particular geometrically bounded \footnote{We thank E.~Shelukhin and M.~Usher for addressing this ``geometrically bounded'' hypothesis to our attention and pointing out a related error in the earlier preprint of our work in subsection \ref{ss-pl}.}, the Chekanov-Hofer (pseudo)-metric on Lagrangian orbits is well-known to be non-degenerate (see \cite{Che00} or \cite{Ush14}). This implies the desired result. 

To prove \eqref{eq-lower-bd}, let $H_t$ be a contact Hamiltonian generating a contact isotopy $\phi_t$ with $\phi_1(L_1) = L_2$. The lifted isotopy $\tilde{\phi}_t : SM \to SM$ is then generated by the Hamiltonian $\tilde{H}_t(\theta, m) = e^\theta H_t(m)$ for $(\theta,m) \in \R \times M = SM$. Note that, by compactness, there exists $R>0$ such that $\cup_t \tilde{\phi}_t(\hat{L}_1) \subset [-R,R] \times M  \subset SM$. Then, for each time $t$,
\begin{equation}\label{eq-Hbar-leq-H}
\max_{\tilde{\phi}_t (\hat{L}_1 ) } |\tilde{H}_t| \leq e^R \cdot \max_{\phi_t(L_1)} |H_t|.
\end{equation}

Truncating $\tilde{H_t}$, we may obtain a compactly supported Hamiltonian $G_t$ on $SM$ such that, for each $t$,  $G_t = \tilde{H_t}$ in a neighbourhood of $\tilde{\phi_t}(\hat{L}_1)$. In particular, the Hamiltonian isotopy $\psi_t$ generated by $G_t$ satisfies $\psi_t(\hat{L}_1) = \tilde{\phi}_t(\hat{L}_1)$. Thus, computing using the analogous formula to \eqref{eq-delta-Usher} for the Chekanov-Hofer distance $\delta_{CH}$ and using \eqref{eq-Hbar-leq-H}, we get:
\begin{align*}
\delta_{CH}(\hat{L}_1, \hat{L}_2 ) \leq \int_0^1 \max_{\psi_t(\hat{L}_1) } |G_t| dt = \int_0^1 \max_{\tilde{\phi}_t (\hat{L}_1) } |\tilde{H}_t|dt \leq e^R \int_0^1 \max_{\phi_t(L)} |H_t| dt.
\end{align*}
Taking the infimum over all such $H_t$ and setting $C = e^R$ yields \eqref{eq-lower-bd}, which, as noted before, completes the proof.
\end{proof}

\subsection{Hypersurface case}

\begin{prop} Let $N$ be a closed hypersurface of a contact manifold $(M, \xi = \ker\alpha)$, then $\delta_{\alpha}$ is non-degenerate on $\mathcal L(N)$. \end{prop}

We will only give the explicit proof when $N$ divides $M$, i.e., $M\backslash N = M_0 \sqcup M_1$ where $M_i$ is open, non-empty and connected and $\bar{M}_i = M_i \cup N$. The general case can be covered by topological argument via Lemma 3.2 and Lemma 3.3 in \cite{Ush14}. 

\begin{proof} For any $N' (\neq N) \in \mathcal L(N)$, denote $M\backslash N' = M'_0 \sqcup M'_1$. Consider, without loss of generality, the case when $N \backslash N' \neq \emptyset$. Without loss of generality, assume $M_0'$ intersects both $M_0$ and $M_1$ non-trivially.
Now let $\phi \in \Cont_0(M,\xi)$ such that $\phi(N) = N'$. Then $\phi$ maps the connected components of $M \setminus N$ to those of $M \setminus N'$, so in particular $M_1'$ is equal to one of $\phi(M_0)$ and $\phi(M_1)$. In the former case, $\phi(M_0) \cap M_0' = \emptyset$ and in particular $\phi$ displaces the (non-empty) open set $M_0' \cap M_1$. In the latter case, similarly, $\phi$ displaces $M_0' \cap M_1$. Therefore, for any such $\phi$, using \eqref{eq-disp-positive},
\[
||\phi||_\alpha \geq \min\{E_\alpha (M_0' \cap M_1), E_\alpha (M_0' \cap M_1) \}>0.
\]
It follows that $\delta_\alpha(N, N')>0$, which completes the proof.
 \end{proof}

\section{Rigid locus} \label{sec-7}
\subsection{Proof of Proposition \ref{rl-c}}

As observed in Proposition 2.2 in \cite{Ush14}, the stabilizer subgroup of $N$, that is, $\Sigma_N = \{\phi \in \Cont_0(M, \xi) \,| \, \phi(N) = N\}$ has a natural topological extension under $||-||_{\alpha}$, that is 

\begin{dfn} \label{closure}
	Denote the closure of $\Sigma_N$ under $\delta_{\alpha}$ by 
	\begin{equation}
	\bar{\Sigma}_N = \{\phi \in \Cont_0(M, \xi) \,| \, \delta_{\alpha}(N, \phi(N)) =0 \}.
	\end{equation}
\end{dfn}
The equality in Definition \ref{closure} follows from first part of the proof of Proposition 2.2 in \cite{Ush14}.

\begin{lemma} \label{non-deg-a}
The $\alpha$-metric	$\delta_\alpha$ is non-degenerate on $\cL(N)$ if and only if $\Sigma_N = {\bar \Sigma}_N$.
\end{lemma}

\begin{proof} If there exists some $\phi \in {\bar \Sigma}_N \backslash \Sigma_N$ (hence $\phi(N) \neq N$), then by Definition (\ref{closure}), $\delta_{\alpha}(N, \phi(N)) = 0$, so $\delta_{\alpha}$ is degenerate. Conversely, suppose that $\bar{\Sigma}_N=\Sigma_N$ and let $L_1, L_2 \in \mathcal L(N)$ such that $\delta_\alpha(L_1, L_2) =0$. Then there exist some $\phi_1, \phi_2 \in \Cont(M, \xi)$ such that 
	\[ 0 = \delta_{\alpha}(L_1, L_2) = \delta_{\alpha}(\phi_1(N), \phi_2(N)) \geq C_-(\phi_1) \cdot \delta_{\alpha}(N, \phi_1^{-1} \phi_2 (N)). \]
	Since $C_-(\phi_1)$ is positive, we get 
	\[ \delta_{\alpha}(N, \phi_1^{-1} \phi_2 (N)) = 0 \]
	which by definition \ref{closure} implies $\phi_1^{-1} \phi_2 \in \bar{\Sigma}_N = \Sigma_N$. Hence, $\phi_1^{-1} \phi_2 (N) = N$ which is equivalent to $L_1 = L_2$. That is, $\delta_\alpha$ is non-degenerate.
\end{proof}

It is easy to see that (1) in Proposition \ref{rl-c} follows directly from Lemma \ref{non-deg-a}. We will focus on the proof of (2) in Proposition \ref{rl-c}. It comes from the following basic lemma which follows almost immediately from Definition \ref{dfn-rl}.

\begin{lemma} \label{frag-lemma}
For any $x \in M \backslash R_N$, there exists a neighborhood $U_x$ of $x$ such that $\Cont_0(U_x) \subset \bar{\Sigma}_N$, where $\Cont_0(U_x)$ is the group of contactomorphisms of $M$ compactly supported in $U_x$. \end{lemma}

\begin{proof} Since $x \in M \backslash R_N$, by definition, there exists some $\phi_x\in\bar{\Sigma}_N$ such that $\phi_x (x) \notin N$. Moreover, since $N$ is closed, there exists a neighborhood $U_x$ such that $\phi_x(U_x) \cap N = \emptyset$. Note that, in particular, $\Cont_0(\phi_x(U_x)) \subset \Sigma_N \subset \bar{\Sigma}_N$. Meanwhile, we know that $\phi_x \Cont_0(U_x) \phi_x^{-1} = \Cont_0(\phi_x (U_x))$. Then since $\bar{\Sigma}_N$ is a group, we know $\Cont_0(U_x) \subset \bar{\Sigma}_N$. \end{proof}

\begin{proof} [Proof of (2) in Proposition \ref{rl-c}] Since $R_N = \emptyset$, according to Lemma \ref{frag-lemma}, for each point $x \in  N$ we can find a neighborhood $U_x$ with condition that $\Cont_0(U_x) \subset \bar{\Sigma}_N$. In this way, we obtain an open cover of $M$, $(M \backslash N) \cup \bigcup_{x \in N} U_x$. According to the contact fragmentation lemma \cite{Ban97, Ryb10}, the group $\Cont_0(M, \xi)$ is generated by contactomorphisms supported in elements of the cover. Since $\Cont_0(U_x) \subset \bar{\Sigma}_N$ for all $x \in N$ by construction, and $\Cont_0(M \setminus N) \subset \Sigma_N$ trivially, we obtain $\Cont_0(M, \xi) = \bar{\Sigma}_N$. Therefore, $\delta_{\alpha}$ vanishes identically. \end{proof}

\begin{remark} \label{frag-rmk} The same argument where the fragmentation lemma is applied to the contact manifold $M \backslash R_N$ implies $\Cont_0(M \backslash R_N) \subset \bar{\Sigma}_N$. This is a parallel result to Proposition 2.1 in \cite{Ush15}. \end{remark}

\subsection{Useful corollaries} 
As a corollary of Remark \ref{frag-rmk}, we can show the following result which claims that a rigid locus $R_N$ behaves very similarly to a contact coisotropic submanifold. However, in general, $R_N$ can be very singular. 

\begin{prop} \label{rl-lie} Let $I_{R_N} = \{H \in C^{\infty}(M) \,| \, H|_{R_N} =0\}$ then $I_{R_N}$ is a Lie subalgebra, i.e., $\{F, G\}_{\alpha} = 0$ on $R_N$ if $F, G \in I_{R_N}$. \end{prop}

\begin{proof} In fact, we only need to show that if $G|_{R_N}=0$, then $\phi_G^t \in \bar{\Sigma}_N$ for any $t$ (only need $t$ sufficiently small). Then indeed, $\phi_G^t$ preserves $R_N$. So for any $x \in R_N$, 
\[ \{F, G\}_{\alpha}(x) = dF(X_G)(x) - dG(R_{\alpha}) F(x) = dF(X_G)(x) = 0. \]
Without loss of generality, assume $t=1$. In order to show $\phi_G^1 \in \bar{\Sigma}_N$, by Remark \ref{frag-rmk}, we only need to show $\phi_G^1$ can be approximated under $\delta_{\alpha}$ by a sequence of contactomorphisms $\phi_n \in \Cont_0(M \backslash R_N)$ (hence in $\bar{\Sigma}_N$). The construction of this sequence is carried out in the proof of Proposition 2.2 in \cite{Ush15} and also in the proof of Lemma 4.3 in \cite{Ush14}. 

Explicitly, take a sequence of smooth functions $\beta_n: \R \to \R$ such that $\beta_n(s) = s$ for $|s| \geq 1/n$ and $\beta_n(s) = 0$ for $|s| < 1/(2n)$. Then note that 
\begin{equation} \label{delta-0}
\max_M |\beta_n \circ G - G| \to 0  \,\,\,\,\,\mbox{as}\,\,\,\,\, n \to \infty.
\end{equation}
Meanwhile, by definition, 
\[ d_{\alpha}(\phi^1_G, \phi^1_{\beta_n \circ G}) = ||\phi^{-1}_G \phi^1_{\beta_n \circ G} ||_\alpha \]
and by composition formula (see the third relation in Lemma 2.2 in \cite{MS15}), we know the contact isotopy $\phi^{-t}_{G} \phi^{t}_{\beta_n \circ G}$ is generated by the contact Hamiltonian 
\[ e^{-g_t} \cdot ((\beta_n \circ G - G) \circ \phi_G^t) \]
where importantly the conformal factor $g_t$ comes from the contact Hamiltonian $G$ which is independent of the sequence $\{\beta_n\}_n$. Therefore
\begin{align*}
||\phi^{-1}_G \phi^1_{\beta_n \circ G} ||_\alpha & \leq \int_0^1 \max_M (e^{-g_t} \cdot |(\beta_n \circ G - G) \circ \phi_G^t|) dt \\
& \leq \int_0^1 \max_M e^{-g_t} \cdot \max_M |(\beta_n \circ G - G) \circ \phi_G^t| dt\\
& \leq  \int_0^1 \max_M e^{-g_t} \cdot \max_M |(\beta_n \circ G - G)| dt\\
& \leq  \left(\int_0^1 \max_M e^{-g_t}dt\right) \cdot \max_M |(\beta_n \circ G - G)| \to 0 \,\,\,\,\mbox{as}\,\,\,\, n \to \infty.
\end{align*}
Therefore, $\phi_{\beta_n \circ G}^1 \xrightarrow{\delta_{\alpha}} \phi_G^1$. Moreover, by construction and our assumption that $G|_{R_N} =0$, each $\phi_{\beta_n \circ G}^1 \in \Cont_0(M \backslash R_N)$ (hence in $\bar{\Sigma}_N$). Therefore, due to closure under $\delta_{\alpha}$, $\phi_G^1 \in \bar{\Sigma}_N$. 
\end{proof}

The following proposition, as a corollary of Proposition \ref{rl-c}, justifies that our interest mostly lies in contact coisotropic submanifolds. 

\begin{prop}\label{nec}
Let $N \subset M$ be a submanifold. If $\delta_{\alpha}$ is non-degenerate on $\mathcal L(N)$, then $N$ is contact coisotropic.
\end{prop}

\begin{proof} We will prove it by contrapositive. Suppose $N$ is not contact coiso\-tropic, then by definition there exists a point $p \in N$ and a vector $v \in (T_pN \cap \xi_p)^{\perp_{d\alpha}}$ which is not in $T_pN \cap \xi_p$. Then we can find a function $H$ on $M$ such that $H|_N =0$ but $dH_p(v) \neq 0$. If we take the same approximating sequence $H_n = \beta_n \circ H \xrightarrow{C^0} H$ as in the proof of Proposition \ref{rl-lie}, then $\phi_{H_n}^t \in \Sigma_N$ and also $\phi_{H_n}^t \xrightarrow{\delta_{\alpha}} \phi_H^t$. Therefore, $\phi_H^t \in \bar{\Sigma}_N$. By definition of $R_N$, $\phi_H^t(R_N) \subset R_N$. Now, for $p \in N$,  $\alpha_p(X_H(p)) = H(p) =0$, so $X_H(p) \in \xi_p$. Meanwhile, for the vector $v$ chosen earlier, 
\[ d\alpha_p(X_H(p), v) = dH(p)(v) \neq 0 \]
which yields $X_H(p) \notin ((T_pN \cap \xi_p)^{\perp_{d\alpha}})^{\perp_{d\alpha}} = T_pN \cap \xi_p$. Therefore, $X_H(p) \notin T_p N$. Then for sufficiently small $t > 0 $, $\phi_H^t(p) \notin N$. Then $\phi_H^t(p) \notin R_N$ because $R_N \subset N$. We conclude $p \notin R_N$ and thus $R_N \neq N$ (so strictly contained in $N$). By Proposition \ref{rl-c}, $\delta_{\alpha}$ is degenerate.  \end{proof}

\section{Dichotomy}
 
Recall the dichotomy phenomena appearing in the Corollary 2.7 in \cite{Ush15} says that in the symplectic set-up, when $\dim N = \frac{1}{2} \dim M$ where $N$ is connected and closed, the Chekanov-Hofer (pseudo)-metric $\delta_{CH}$ is either non-degenerate or vanishes identically. Theorem \ref{dic} shows an analogue result in the contact topology set-up.

\subsection{Local model analysis} The proof of Theorem \ref{dic} starts from the following local analysis. Let $M$ be a contact manifold with dimension $2n+1$ and $N \subset M$ be a submanifold of dimension $n$. For any $x \in N$, its neighborhood can be modeled as a neighborhood of $\vec{0}$ in $(\R^{2n+1}, \alpha)$ where in coordinate $(x_1, \dots, x_n, y_1,\dots, y_n, z)$, $\alpha = dz - \sum y_i dx_i$. Without loss of generality, we can assume 
either 
\[ N = \{x_1= \cdots = x_k = y_1 = \cdots y_{n-k} = z =0\}\]
or 
\[ N = \{x_1= \cdots = x_{k+1} = y_1 = \cdots =y_{n-k} =0\} \]
We will careful study the first case and the second case will be discussed in Remark \ref{one-more}. Also note that the first case covers the situation where $N$ is a Legendrian submanifold by standard neighborhood theorem. Consider the following three types of coordinate-functions near $\vec{0}$. (i) $F = x_m$ for $m \in \{1, \dots, k\}$; (ii) $F = y_m$ for $m \in \{1 ,\dots, n-k\}$; (iii) $F = z$. We will investigate their corresponding contact vector fields. Recall that the contact vector field $X_F$ is uniquely determined by the following differential equations
\begin{equation} \label{chv}
\left\{ \begin{array}{ll} \iota_{X_F} d\alpha = dF(R_{\alpha}) \alpha - dF\\ \alpha(X_F) = F \end{array} \right. 
\end{equation}
where locally $R_{\alpha} = \frac{\partial}{\partial z}$ and $d \alpha = \sum dx_i \wedge dy_i$. Let $X_F = \sum A_i \frac{\partial}{\partial x_i} + B_i \frac{\partial}{\partial y_i} + C \frac{\partial}{\partial z}$,
\begin{itemize}
\item[(i)] if $F = x_m$, we know $dF(R_{\alpha}) = 0$, so we are reduced to solve 
\begin{equation} \label{chv-1}
\left\{ \begin{array}{ll} \iota_{X} d\alpha = - dF\\ \alpha(X) = F \end{array} \right.  \Longrightarrow \left\{ \begin{array}{ll} \sum (A_i dy_i - B_i dx_i) = - dx_m \\ C - \sum A_iy_i = x_m \end{array} \right. 
\end{equation}
that is, $A_i = 0$ for all $i$, $B_i  = 1$ only for $i=m$ and $0$ otherwise, and $C = x_m$. Therefore, $X_F = \frac{\partial}{\partial y_m}  + x_m \frac{\partial}{\partial z}$. Note that at $\vec{0}$, $X_F(\vec{0}) = \frac{\partial}{\partial y_m} \neq 0$;
\item[(ii)] if $F = y_m$, we still know $dF(R_{\alpha}) = 0$, so we are reduced to solve 
\begin{equation} \label{chv-2}
\left\{ \begin{array}{ll} \iota_{X} d\alpha = - dF\\ \alpha(X) = F \end{array} \right.  \Longrightarrow \left\{ \begin{array}{ll} \sum (A_i dy_i - B_i dx_i) = - dy_m \\ C - \sum A_iy_i = y_m \end{array} \right. 
\end{equation}
that is, $A_i = -1$ only for $i=m$ and $0$ otherwise, $B_i  = 0$ for all $i$, and $C = 0$. Therefore, $X_F = -\frac{\partial}{\partial x_m}$ and at $\vec{0}$, $X_F(\vec{0}) \neq 0$.  
\item[(iii)] if $F = z$, we know $dF(R_{\alpha}) = 1$, so we are reduced to solve 
\begin{equation} \label{chv-3}
\left\{ \begin{array}{ll} \iota_{X} d\alpha = \alpha - dF\\ \alpha(X) = F \end{array} \right.  \Longrightarrow \left\{ \begin{array}{ll} \sum (A_i dy_i - B_i dx_i) = -\sum y_i dx_i \\ C - \sum A_iy_i = z \end{array} \right. 
\end{equation}
that is, $A_i = 0$ for all $i$, $B_i  = -y_i$ only for all $i$, and $C = z$. Therefore, $X_F = \sum y_i\frac{\partial}{\partial y_i}  + z \frac{\partial}{\partial z}$. Note that at $\vec{0}$, $X_F$ is degenerate, i.e. $X_F(\vec{0}) = 0$.
\end{itemize}
Let $F_i$ denote the coordinate-function of type (i) and (ii) above (and there are in total $n$ many of them). For any fixed point $x \in N$, the following map 
\[ \Phi: \R^{n} \to M \,\,\,\,\mbox{by}\,\,\,\, (a_1, \dots, a_n) \to \phi_{\sum a_i F_i}^1(x) \]
provides an embedding 
\begin{equation} \label{emb}
B^n(\ep) \hookrightarrow \mbox{neighborhood of $x$ in $N$}
\end{equation}
where $B^n(\ep)$ is the open $n$-dimensional ball with radius $\ep>0$. This is because $d\Phi(\vec{0})$ maps each basis element $e_i$ to a vector $X_{F_i}(x)$ - the contact Hamiltonian vector field of $F_i$ at $x$ - for some $i$. By computations in (\ref{chv-1}) and (\ref{chv-2}), these $n$ vectors are linearly independent. Be aware that we abandoned $F=z$ in type (iii) above because at $\vec{0}$, it is degenerate (so it does not provide a linearly independent direction contributing to our embeddings). In other words, $n$ is the maximal dimension that we can obtain for embeddings like $(\ref{emb})$. 

\begin{remark} \label{one-more} 
Recall that the local model of $N$ have two different choices, and their difference is whether $\{z =0\}$ is included or not. Based on the computation above, observe that we don't have the type (iii) if one considers the local model $N$ without the condition $\{z=0\}$. Then by the same argument as above, we have $(n+1)$-many linearly independent vectors coming from $(n+1)$-many coordinate-functions which do not include $z$. Hence, for any fixed $x \in N$, there exists an embedding $B^{n+1}(\ep)$ into a neighborhood of $x$ in~$N$. 
\end{remark}

\subsection{Proof of Theorem \ref{dic}}
\begin{proof} [Proof of Theorem \ref{dic}]
Suppose $R_N \neq \emptyset$. Then for any $x \in R_N$, by (\ref{emb}) and Remark \ref{one-more}, there exists an embedding from either $B^{n}(\ep)$ or $B^{n+1}(\ep)$ into $N$. By our choice of coordinate-functions, in either case, $\sum a_i F_i \in I_{R_N}$ because $R_N \subset N$ and each $F_i|_N=0$. By the proof of Proposition \ref{rl-lie}, $\phi_{\sum a_i F_i}^1 \in \bar{\Sigma}_N$. In particular, $\phi_{\sum a_i F_i}^1$ preserves $R_N$. In other words, we can upgrade the embedding from $B^n(\ep)$ or $B^{n+1}(\ep)$ into $R_N$. 

In terms of embedding from $B^{n+1}(\ep)$, it already gives a contradiction because $\dim N = n< n+1$. On the other hand, due to the embedding from $B^n(\ep)$, one knows $R_N$ is open in $N$. Meanwhile, $R_N$ is also closed by its definition, which, by the connectedness of $N$, implies $R_N = N$. Therefore, by Proposition \ref{rl-c}, $\delta_{\alpha}$ is non-degenerate. \end{proof}

\section{Different measurements} \label{sec-dm}
One standard way to approach Conjecture \ref{conj-1} (at least following the original idead from \cite{Che00}, or its extension in Section 4 in \cite{Ush14} and \cite{Ush15} in the symplectic set-up) is first getting a dichotomy result as above and then using some energy-capacity inequality to rule out the identical vanishing possibility. Interested readers can check a successful procedure in this spirit from Theorem 4.9 and Corollary 4.10 in \cite{Ush14} (or the original argument in \cite{Che00}). Naive attempts to implement this approach in proving Conjecture \ref{conj-1}  run into the following problem: the energy estimates in the symplectic case stem from the positivity of displacement energy of Lagrangians, which fails for Legendrians in view of Example \ref{ex-inf-dis}. One possible solution is to look for displacement of sets in the symplectization, and apply Proposition 11 in \cite{She17} rather then Corollary 15. 
Here the main difficulty is the lack of a well-established energy-capacity inequality (only) in terms of norm $||-||_{\alpha}$. Note the proof of Proposition 10 in \cite{She17} still involves conformal factors (but well-controlled for a {\it single} contactomorphism via a cut-off technique). However, when taking infimum over various contactomorphisms, for instance, in the definition of $\delta_{\alpha}$ considering all the contactomorphisms $\phi$ such that $\phi(L_1) = L_2$, there is no guarantee that in the symplectization there exists a uniform ball which can be displaced.

In this section, we want to point out that in the contact topology set-up, unlike Hofer's metric in symplectic topology, there is no such ``canonical'' quantity to measure the behavior of dynamics. Instead, several quantities have been invented to get certain rigidity results in contact topology. For instance, the following one was used in \cite{RS16}, namely, for $\phi \in \Cont_0(M, \xi)$ set
\begin{equation}\label{eq-RS-norm}
||\phi||_{RS} : = \inf \left\{e^{- \min_{M, t \in [0,1]} g_t} \left(\int_0^1 \max_M H(\cdot, t) - \min_{M} H(\cdot, t) dt\right)\right\},
\end{equation}
where the infimum is taken over all contact Hamiltonian $H$ generating a contact isotopy $\phi_H^t$ with $\phi_H^1=\phi$, and $g_t$ is the conformal factor of $\phi_H^t$ for $t \in [0,1]$ with respect to the contact 1-form $\alpha$. It would be an interesting question to compare $||\phi||_{RS}$ with $||\phi||_{\alpha}$. In that regard, let us mention that removing the conformal term from \eqref{eq-RS-norm} one obtains the definition of Shelukhin's oscillation semi-norm, which is related to the $\alpha$-distance from the Reeb subgroup (see Definition 23 and Proposition 24 in \cite{She17}.)

In our context, we can consider a related quantity, a modification of the Hofer-Shelukhin norm which is sensitive to divergence of conformal factors (this is suggested by M. Usher).

\begin{dfn} \label{dfn-norm-mod} Given a contact manifold $(M, \xi = \ker \alpha)$, for any $\phi \in \Cont_0(M, \xi)$, define 
\[ ||\phi||_{\alpha, m} : = ||\phi||_{\alpha} + \max_M |g_{\phi}| \]
where $g_{\phi}$ is the conformal factor of $\phi$ with respect to contact 1-form $\alpha$, i.e. $\phi^*\alpha = e^{g_\phi}\alpha$. \end{dfn}

Then one easily gets the following proposition. 

\begin{prop}
The quantity $||\phi||_{\alpha,m}$ defined in Definition \ref{dfn-norm-mod} satisfies 
\begin{itemize}
    \item[(i)] $||\phi||_{\alpha,m} \geq ||\phi||_{\alpha}$ and equality holds for any strict contactomorphism. 
    \item[(ii)] $||\phi||_{\alpha, m}$ is a norm on $\Cont_0(M, \xi)$.
\end{itemize}
\end{prop}

\begin{proof} (i) is trivial since for any strict contactomorphism, its associated conformal factor is always $0$. (ii) follows from properties of conformal factors. If $g_{\phi}$ is the conformal factor of $\phi$, then $- g_{\phi}\circ \phi^{-1}$ is the conformal factor of $\phi^{-1}$, so 
\[ ||\phi^{-1}||_{\alpha, m} = ||\phi^{-1}||_{\alpha} + \max_{M} |- g_{\phi} \circ \phi^{-1}| = ||\phi||_{\alpha} + \max_{M}|g_{\phi}|= ||\phi||_{\alpha, m}. \]
Similarly, if $g_{\phi}$ and $g_{\psi}$ are conformal factors of $\phi$ and $\psi$ respectively, then $g_{\psi} \circ \phi + g_{\phi}$ is the conformal factor of $\psi \circ \phi$, so 
\begin{align*}
    ||\psi \circ \phi||_{\alpha, m} & = ||\psi \circ \phi||_{\alpha} + \max_{M} |g_{\psi \circ \phi}| \\
    & \leq ||\psi||_{\alpha} +||\phi||_{\alpha} + \max_M |g_{\psi} \circ \phi + g_{\phi}|\\
    & \leq (||\psi||_{\alpha} + \max_{M} |g_{\psi}|) + (||\phi||_{\alpha} + \max_{M} |g_{\phi}|) = ||\psi||_{\alpha,m} + ||\phi||_{\alpha, m}.
\end{align*}
Finally, the non-degeneracy of $||-||_{\alpha, m}$ trivially comes from the non-degeneracy of $||-||_{\alpha}$ and (i). Moreover, when $\phi = \mathds{1}$, $||\mathds{1}||_{\alpha, m} = ||\mathds{1}||_{\alpha} =0$.
\end{proof}

\begin{remark} $||-||_{\alpha, m}$ provides a measurement which is comparable with (but not equivalent to) $||-||_{\alpha}$. Similar to $\delta_{\alpha}$, by using $||-||_{\alpha, m}$, we can also define Chekanov-type psuedo-metic denoted as $\delta_{\alpha, m}$. Interested readers can check that all the conclusions on $\delta_{\alpha}$ also hold for $\delta_{\alpha, m}$ (thus we also have a dichotomy with respect to this modified norm). Certainly Conjecture \ref{conj-1} can be modified to be stated under $\delta_{\alpha,m}$. From our personal perspective, this modified conjecture is more likely to be true. Finally, for some rigidity results characterized by quantity like $||-||_{RS}$ involving Legendrian submanifolds, see Section 1.3 in \cite{RS16}.  \end{remark}

\bibliography{9-ZhangJ}
\noindent \\

\end{document}